\documentclass{amsart}

\usepackage{amssymb}
\usepackage{amsthm}
\usepackage{amsfonts}
\usepackage{mathrsfs}
\usepackage{amsmath}
\usepackage{extarrows}
\usepackage{color}
\newtheorem{theorem}{Theorem}[section]
\newtheorem{corollary}[theorem]{Corollary}
\newtheorem{lemma}[theorem]{Lemma}

\numberwithin{equation}{section}

\newcommand{\RePt}{\mathrm{Re}\,}
\newcommand{\ImPt}{\mathrm{Im}\,}
\newcommand{\ball}{\mathbb{B}}

\newcommand{\calU}{\mathcal{U}}

\newcommand{\calS}{\mathcal{S}}
\newcommand{\bfi}{\mathbf{i}}

\newcommand{\bfrho}{\boldsymbol{\rho}}

\newcommand{\bbB}{\mathbb{B}}
\newcommand{\bbC}{\mathbb{C}}

\begin{document}

\title[Positive Toeplitz operators]{Positive Toeplitz operators on the Bergman spaces of the Siegel upper half-space}
\thanks{This work was supported by the National Natural Science Foundation of China grants 11571333.}


\author{Congwen Liu}
\email{cwliu@ustc.edu.cn}

\address{School of Mathematical Sciences,
University of Science and Technology of China,\\
Hefei, Anhui 230026,
People's Republic of China\\
and\\
Wu Wen-Tsun Key Laboratory of Mathematics\\
USTC, Chinese Academy of Sciences}

\author{Jiajia Si}
\email{sijiajia@mail.ustc.edu.cn}

\address{School of Mathematical Sciences,
University of Science and Technology of China,\\
Hefei, Anhui 230026,
People's Republic of China}


\begin{abstract}
We characterize bounded and compact positive Toeplitz operators defined on the Bergman spaces over the Siegel upper half-space.
\end{abstract}


\maketitle

\section{Introduction}

Toeplitz operators on Bergman spaces over the unit disk have been well studied.
Especially, positive symbols of bounded and compact Toeplitz operators
are completely characterized. See for instance \cite{Axl88} or \cite[Chapter 7]{Zhu07}. These results were further extended to more general settings
in \cite{Zhu88} and \cite{Yam13}. However, there are only few works on analogous results over unbounded domains. See \cite{CKY02,CN07}
for a study of positive Topelitz operators on harmonic Bergman spaces over the upper half space of $\mathbb{R}^n$.

In this paper we study bounded and compact positive Toeplitz operators on Bergman spaces over the Siegel upper half-space.

Let $\mathbb{C}^n$ be the $n$-dimensional complex Euclidean space.
For any two points $z=(z_1,\cdots,z_n)$ and $w=(w_1,\cdots,w_n)$ in $\mathbb{C}^n$ we write
\[
z\cdot\overline{w} := z_1\overline{w}_1+\cdots+z_n\overline{w}_n,
\]
and
\[
|z|:= \sqrt{z\cdot\overline{z}} = \sqrt{|z_1|^2+\cdots+|z_n|^2}.
\]
The unit ball of $\bbC^n$ is given by
\[
\ball=\{ z\in\bbC^n: |z|<1\}.
\]
The set
\[
\calU=\left\{ z\in\mathbb{C}^n:\ImPt z_n>|z^{\prime}|^2\right\}
\]
is the Siegel upper half-space. Here and the throughout the paper, we use the notatoin
\[
z=(z^{\prime},z_n),\,\,\,\, \text{where}\, z^{\prime}=(z_1,\cdots,z_{n-1})\in\mathbb{C}^{n-1}\,\, \text{and}\,\, z_n\in\mathbb{C}^{1}.
\]

As usual, for $p>0$, the space $L^p(\calU)$ consists of all Lebesgue measurable
functions $f$ on $\calU$ for which
\begin{equation*}
\|f\|_p:=\bigg\{\int\limits_{\calU} |f(z)|^p dV(z)\bigg\}^{1/p}
\end{equation*}
is finite, where $V$ denotes the Lebesgue measure on $\mathbb{C}^n$.
The Bergman space $A^p(\calU)$ is the closed subspace of $L^p(\calU)$ consisting of
holomorphic functions on $\calU$.
Note that when $1\leq p<\infty$ the space $A^p(\calU)$ is a Banach space with the norm $\|\cdot\|_p$. In particular, $A^2(\calU)$ is a Hilbert space
endowed with the usual $L^2$ inner product. The orthogonal projection from $L^2(\calU)$ onto $A^2(\calU)$ can be expressed as an integral operator:
\[
Pf(z)=\int\limits_{\calU} K(z,w)f(w)dV(w),
\]
with the Bergman kernel
\[
K(z,w)=\frac{n!}{4\pi^n}\left[ \frac{i}{2}(\overline{w}_n-z_n)-z^{\prime} \cdot \overline{w^{\prime}} \right]^{-n-1}.
\]
See, for instance, \cite[Theorem 5.1]{Gin64}.
This formula enables us to extend the domain of the operator $P$, which is usually
called a Bergman projection, to $L^p(\calU)$ for all $1<p<\infty$. Moreover,
$P$ is a bounded projection from $L^p(\calU)$ onto $A^p(\calU)$ for $1<p<\infty$, see \cite[Lemma 2.8]{CR80}.

Given $\varphi \in L^{\infty}(\calU)$, we define an operator on $A^p(\calU)$ by
\[
T_{\varphi} f := P(\varphi f), \quad f\in A^p(\calU).
\]
$T_{\varphi}$ is called the Toeplitz operator on $A^p(\calU)$ with symbol $\varphi$.
Toeplitz operators can also be defined for unbounded symbols or even positive Borel measures on $\calU$.
Let $\mathcal{M}_+$ be the set of all positive Borel measures $\mu$ such that
\[
\int\limits_{\calU}\frac{d\mu(z)}{|z_n+i|^{\alpha}}<\infty
\]
for some $\alpha>0$.
Given $\mu\in \mathcal{M}_+$, the Toeplitz operator $T_{\mu}$ with symbol $\mu$ is given by
\[
T_{\mu}f(z)=\int\limits_{\calU} K(z,w)f(w) d\mu(w)
\]
for $f\in H(\calU)$. In general, $T_{\mu}$ may not even be defined on all of $A^p(\calU)$, $1<p<\infty$,
but it is always densely defined by the fact that, for each $\alpha>n+1/p$, holomorphic functions $f$ on $\calU$
such that $f(z) = O(|z_n+i|^{-\alpha})$ form a dense subset of $A^p(\calU)$ (see Section 4).
This is inspired by that of Choe et al.\,\cite{CKY02}, in the setting of the upper half space of $\mathbb{R}^n$.

A positive Borel measure $\mu$ is called a Carleson measure for the Bergman space $A^p(\calU)$ if there exists a positive constant $C$ such that
\[
\int\limits_{\calU} |f(z)|^p d\mu(z) \leq C \int\limits_{\calU} |f(z)|^p dV(z)
\]
for all $f\in A^p(\calU)$.  We shall furthermore say that $\mu$ is a vanishing Carleson measure for $A^p(\calU)$ if the inclusion map
$i_p: A^p(\calU)\to L^p(\calU,\mu),  f\mapsto f$  is compact, that is,
\[
\lim_{j\to\infty} \int\limits_{\calU} |f_j(z)|^p d\mu(z)=0
\]
whenever $\{f_j\}$ converges to $0$ weakly in $A^p(\calU)$.


For a positive Borel measure $\mu$ on $\calU$, we formally define a function
$\widetilde{\mu}$ on $\calU$ by
\begin{equation}\label{eqn:berezin}
\widetilde{\mu}(z) := \int\limits_{\calU} |k_z(w)|^2 d\mu(w),\quad z\in \calU,
\end{equation}
where, for $z\in \calU$,
\[
k_z(w) := K(z,w)/\sqrt{K(z,z)},\quad w\in\calU,
\]
and call $\widetilde{\mu}$ the Berezin transform of $\mu$.
For $z\in\calU$ and $r>0$, we define the averaging function
\begin{equation}\label{eqn:averaging}
\widehat{\mu}_r(z) := \frac{\mu(D(z,r))}{|D(z,r)|},
\end{equation}
where $D(z,r)$ is the  Bergman metric ball at $z$ with radius $r$ (see Section \ref{subsec:Bergman metric balls}) and
$|D(z,r)|:=V(D(z,r))$ denotes the Lebesgue measure of $D(z,r)$.

We can now state our main results.

\begin{theorem}\label{thm:T_mu,bounded}
Suppose that $r>0$, $1<p<\infty$, $0<q<\infty$ and that $\mu\in \mathcal{M}_+$. Then the following conditions are equivalent:
\begin{enumerate}
  \item[(\romannumeral1)] $T_{\mu}$ is bounded on $A^p(\calU)$.
  \item[(\romannumeral2)] $\widetilde{\mu}$ is a bounded function on $\calU$.
  \item[(\romannumeral3)] $\widehat{\mu}_r$ is a bounded function on $\calU$.
  \item[(\romannumeral4)] $\mu$ is a Carleson measure for $A^q(\calU)$.
\end{enumerate}
\end{theorem}

Let $\bfrho(z):= \ImPt z_n -|z^{\prime}|^2$ and $b\calU:=\{z\in \mathbb{C}^n: \bfrho(z) = 0\}$ denotes the boundary of $\calU$. Then
$\widehat{\calU} := \calU \cup b\calU \cup \{\infty\}$ is the one-point compactification of $\calU$. Also, let $\partial \widehat{\calU}:=b\calU\cup\{\infty\}$. Thus,
$z\to\partial \widehat{\calU}$ means $\bfrho(z)\to 0$ or $\bfrho(z)\to \infty$.
We denote by $C_0(\calU)$ the space of complex-value continuous functions $f$ on $\calU$ such that $f(z)\to 0$ as $z\to\partial \widehat{\calU}$.

\begin{theorem}\label{thm:T_mu,compact}
Suppose that $r>0$, $1<p,q<\infty$ and that $\mu\in \mathcal{M}_+$. Then the following conditions are equivalent:
\begin{enumerate}
  \item[(\romannumeral1)] $T_{\mu}$ is compact on $A^p(\calU)$.
  \item[(\romannumeral2)] $\widetilde{\mu}$ belongs to $C_0(\calU)$.
  \item[(\romannumeral3)] $\widehat{\mu}_r$ belongs to $C_0(\calU)$.
  \item[(\romannumeral4)] $\mu$ is a vanishing Carleson measure for $A^q(\calU)$.
\end{enumerate}
\end{theorem}

This paper is organized as follows. Section 2 contains the necessary background material and auxiliary results.
In Section 3, we characterize Carleson measures and vanishing Carleson measures for the Bergman spaces over the Siegel
upper half-space. In Section 4, we show that the Toeplitz operators are well defined on a family of dense subspaces
of the Bergman spaces.
The proofs of the theorems \ref{thm:T_mu,bounded} and \ref{thm:T_mu,compact} are carried out in Section 5.

Throughout the paper, the letter $C$ will denote a positive constant that may vary at each occurrence but is independent of the essential variables.
The letter $C$ with subscripts usually denotes a specific constant.

\section{Preliminaries}

\subsection{Estimates of the Bergman kernel}
For simplicity, we write
\[
\bfrho(z,w):=\frac{i}{2}(\overline{w}_n-z_n)-z^{\prime} \cdot \overline{w^{\prime}}.
\]
With this notation, the Bergman kernel of $\calU$
\[
K(z,w) = \frac {n!}{4\pi^n} \frac {1}{\bfrho(z,w)^{n+1}}, \quad z,w\in \calU.
\]
Note also that $\bfrho(z)=\bfrho(z,z)$.

\begin{lemma}
We have
\begin{equation}\label{eqn:elemtryeq1}
 |K(z,w)| ~\leq~ \frac {2^{n-1}n!}{\pi^n} \min\{\bfrho(z),\bfrho(w)\}^{-n-1}
\end{equation}
for any $z, w\in \calU$.
\end{lemma}

\begin{proof}
For each $t>0$, we define the nonisotropic dilation $\delta_t$ by
\[
\delta_t(u)=(t u^{\prime},t^2 u_n), \quad u\in \calU.
\]
Also, to each fixed $z\in\calU$, we associate the following (holomorphic) affine
self-mapping of $\calU$:
\[
h_z(u) ~:=~ \left(u^{\prime} - z^{\prime}, u_n - \RePt z_n - 2i u^{\prime} \cdot \overline{z^{\prime}}  + i|z^{\prime}|^2 \right),
\quad u\in \calU.
\]
All these mappings are holomorphic automorphisms of $\calU$. See \cite[Chapter XII]{Ste93}. Hence the mappings $\sigma_z := \delta_{\bfrho(z)^{-1/2}} \circ h_z$
are holomorphic automorphisms of $\calU$. Simple calculations show that $\sigma_z(z)=\bfi:=(0^{\prime},i)$ and
\begin{equation}\label{eqn:jacobian}
(J_{\mathbb{C}} \sigma_z) (u) = \bfrho(z)^{-(n+1)/2}
\end{equation}
where $(J_{\mathbb{C}} \sigma_z)(u)$ stands for the complex jacobian of $\sigma_z$ at $u$.

Thus, by \cite[Proposition 1.4.12]{Kra01}, we have
\begin{align}
K(z,w) ~=~& (J_{\mathbb{C}} \sigma_z) (z)\, K(\sigma_z(z),\sigma_z(w))\,  \overline{(J_{\mathbb{C}} \sigma_z) (w)} \label{eqn:Berg}\\
~=~& K(\bfi,\sigma_z(w))\, \bfrho(z)^{-n-1}. \notag
\end{align}
Note that
\[
|K(\bfi,u)| ~=~ \frac {n!}{4\pi^n} \frac {2^{n+1}}{|u_n+i|^{n+1}} ~\leq~ \frac {2^{n-1} n!}{\pi^n}
\]
for all $u\in\calU$. Hence \eqref{eqn:elemtryeq1} follows immediately from \eqref{eqn:Berg}.
\end{proof}

\begin{lemma}\label{lem:normofbergman}
Let $1<p<\infty$. Then for each $z\in \calU$, the Bergman kernel function $K_z:=K(\cdot,z)$ is in $A^p(\calU)$, and
\begin{equation}\label{eqn:normofbergman}
\|K_z \|_{p}  ~=~ C_{n,p}\, \bfrho(z)^{-(n+1)/p^{\prime}},
\end{equation}
where $p^{\prime}=p/(p-1)$ and $C_{n,p}$ is a positive constant depending on $n$ and $p$.
\end{lemma}

\begin{proof}
This is just an application of the following formula from \cite[Lemma 5]{LLHZ18}:
\begin{equation}\label{eqn:keylem}
\int\limits_{\calU} \frac {\bfrho(w)^{t}} {|\bfrho(z,w)|^{s}} dV(w) ~=~
\begin{cases}
\dfrac {C_{n,s,t}} {\bfrho(z)^{s-t-n-1}}, &
\text{ if } t>-1 \text{ and } s-t>n+1\\[12pt]
+\infty, &  otherwise
\end{cases}
\end{equation}
for all $z\in \calU$, where
\[
C_{n,s,t}:=\frac {4 \pi^{n} \Gamma(1+t) \Gamma(s-t-n-1)} {\Gamma^2\left(s/2\right)}.
\]
We omit the details.
\end{proof}

\subsection{Cayley transform and the M\"obius transformations}

Recall that the Cayley transform $\Phi:\bbB \to \calU$ is given by
\[
(z^{\prime}, z_{n})\; \longmapsto\; \left( \frac {z^{\prime}}{1+z_{n}},
i\left(\frac {1-z_{n}}{1+z_{n}}\right) \right).
\]
It is immediate to calculate that
\[
\Phi^{-1}: \left(z^{\prime},z_{n}\right)\;\longmapsto\; \left(\frac {2iz^{\prime}}{i+z_{n}},
\frac {i-z_{n}}{i+z_{n}}\right).
\]
We refer to \cite[Chapter XII]{Ste93} for the properties of the Cayley transform.
For the convenience of later reference, we record the following lemma from \cite{Liu2018}.
\begin{lemma}\label{lem:cayley}
\begin{enumerate}
\item[(i)]
The identity
\begin{equation}\label{eqn:identity14phi}
\bfrho(\Phi(\xi),\Phi(\eta)) = \frac {1-\xi\cdot \overline{\eta}} {(1+\xi_{n}) (1+\overline{\eta}_{n})}
\end{equation}
holds for all $\xi,\eta\in \ball$.
\item[(ii)]
The real Jacobian of $\Phi$ at $\xi\in \ball$ is
\begin{equation}\label{eqn:jacobian4phi}
\left(J_{R}\Phi\right)(\xi) = \frac {4}{|1+\xi_{n}|^{2(n+1)}}.
\end{equation}
\item[(iii)]
The identity
\begin{equation}\label{eqn:cayleyidentity1}
1- \Phi^{-1}(z)\cdot \overline{\Phi^{-1}(w)} = \frac {\bfrho(z,w)} {\bfrho(z,\bfi)\bfrho(\bfi,w)}
\end{equation}
holds for all $z,w\in \calU$, where $\bfi=(0^{\prime},i)$.
\item[(iv)]
The identity
\begin{equation}\label{eqn:cayleyidentity2}
|\Phi^{-1}(z)|^2 = 1- \frac {\bfrho(z)} {|\bfrho(z,\bfi)|^2}
\end{equation}
holds for all $z\in \calU$.
\item[(v)]
The real Jacobian of $\Phi^{-1}$ at $z\in \calU$ is
\begin{equation}\label{eqn:jacobian Phi^-1}
\left(J_{R}\Phi^{-1}\right)(z) = \frac {1}{4|\bfrho(z,\bfi)|^{2(n+1)}}.
\end{equation}
\end{enumerate}
\end{lemma}

The group of all one-to-one holomorphic mappings of $\ball$ onto $\ball$ (the so-called automorphisms
of $\ball$) will be denoted by $\mathrm{Aut}(\ball)$. It is generated by the unitary transformations
on $\bbC^{n}$ along with the M\"obius transformations
$\varphi_{\xi}$ given by
\[
\varphi_{\xi}(\eta) := \frac {\xi-P_{\xi}\eta-(1-|\xi|^2)^{\frac {1}{2}}Q_{\xi}\eta}{1-\eta\cdot\overline{\xi}},
\]
where $\xi\in \ball$, $P_{\xi}$ is the orthogonal projection onto the space spanned
by $\xi$,
and $Q_{\xi}\eta=\eta- P_{\xi}\eta$. See \cite[Section 1.2]{Zhu05}.

It is easily shown that the mapping $\varphi_{\xi}$ satisfies
\[
\varphi_{\xi}(0)=\xi, \quad \varphi_{\xi}(\xi)=0, \quad \varphi_{\xi}(\varphi_{\xi}(\eta))=\eta.
\]
Furthermore, for all $\eta, \omega\in \overline{\ball}$, we have
\begin{equation}
1-   \varphi_{\xi}(\eta) \cdot \overline{\varphi_{\xi}(\omega)} ~=~ \frac {(1- |\xi|^2)
(1- \eta \cdot \overline{\omega})} {(1- \eta\cdot \overline{\xi})(1-  \xi\cdot \overline{ \omega})}. \label{eqn:moeb0}
\end{equation}

\subsection{Bergman metric balls}\label{subsec:Bergman metric balls}
Let $\Omega$ be a domain in $\mathbb{C}^n$ and $K_{\Omega}(z,w)$ be the Bergman kernel of $\Omega$. We define
\[
g^{\Omega}_{i,j}(z) := \frac {1}{n+1} \frac {\partial^2 \log K_{\Omega}(z,z)}{\partial z_i \partial \bar{z}_j},
\qquad i,j=1,\cdots,n
\]
and call the complex matrix
\[
B_{\Omega}(z)=\big(g^{\Omega}_{i,j}(z)\big)_{1\leq i,j \leq n}
\]
the Bergman matrix of $\Omega$. For a $C^1$ curve $\gamma:[0,1]\to\Omega$ we define
\[
l_{\Omega}(\gamma)=\int_{0}^{1}  \left(B_{\Omega}(\gamma(t)) \gamma^{\prime}(t)\cdot \overline{\gamma^{\prime}(t)} \right)^{1/2} dt.
\]
For any two points $z$ and $w$ in $\Omega$, let $ \beta_{\Omega}(z,w)$
be the infimum of the set consisting of all $l_{\Omega}(\gamma)$, where $\gamma$ is a piecewise smooth curve in $\Omega$ from $z$ to $w$.
We will call $\beta_{\Omega}$ the Bergman metric on $\calU$.
For $z\in\Omega$ and $r>0$ we let $D_{\Omega}(z,r)$ denote the Bergman metric ball at $z$ with radius $r$. Thus
 \[
 D_{\Omega}(z,r) := \{w\in\calU: \beta_{\Omega}(z,w)<r\}.
 \]

 If $\Omega_1$, $\Omega_2$ are two domains in $\mathbb{C}^n$ and $h$ is a biholomorphic mapping of $\Omega_1$ onto $\Omega_2$, then
\[
\beta_{\Omega_1}(z,w)=\beta_{\Omega_2}(h(z),h(w))
\]
for all $z,w\in\Omega_1$. See for instance \cite[Proposition 1.4.15]{Kra01}. Hence,
\begin{equation}\label{eq:bergmanmetric}
\beta_{\calU}(z,w)= \beta_{\ball}(\Phi^{-1}(z),\Phi^{-1}(w))=\tanh^{-1}\left( \left|\varphi_{\Phi^{-1}(z)}(\Phi^{-1}(w))\right|\right).
\end{equation}
It follows that
\[
D_{\calU}(z,r) = \Phi(D_{\ball}(\Phi^{-1}(z),r))
\]
for every $z\in\calU$ and $r>0$.
Also, a computation shows that
\begin{equation}\label{eqn:tanh_beta}
 \beta_{\calU}(z,w)=\tanh^{-1}\sqrt{1-\frac{\bfrho(z)\bfrho(w)}{|\bfrho(z,w)|^2}}.
\end{equation}

In the sequel, we simply write $\beta(z,w):=\beta_{\calU}(z,w)$ and $D(z,r):=D_{\calU}(z,r)$ if not cause any confusion.



\begin{lemma}\label{lem:decomposition}
 For any $r>0$, there exists a sequence $\{a_k\}$ in $\calU$ such that
\begin{enumerate}
  \item[(i)]
$\calU=\bigcup_{k=1}^{\infty} D(a_k,r)$;
  \item[(ii)]
There is a positive integer $N$ such that each point $z$ in $\calU$ belongs to at most $N$ of the sets $D(a_k,2r)$.
\end{enumerate}
We are going to call $\{a_k\}$ an $r$-lattice in the Bergman metric.
\end{lemma}
\begin{proof}
The proof is similar to that of \cite[Theorem 2.23]{Zhu05}, so is omitted.
\end{proof}

\begin{lemma}\label{lem:volBergmanball}
For any $z\in\mathcal{U}$ and $r>0$ we have
\begin{equation}
|D(z,r)| ~=~ \frac {4\pi^n}{n!} \frac {\tanh^{2n} r } {(1-\tanh^2 r)^{n+1}}\, \bfrho(z)^{n+1} .
\end{equation}
\end{lemma}

\begin{proof}
We first show that
\begin{equation}\label{eqn:vols}
|D(z,r)| = \bfrho(z)^{n+1} |D(\bfi,r)|
\end{equation}
holds for any $z\in\mathcal{U}$ and $r>0$.

Since the metric $\beta$ is invariant under the automorphisms, we have
\begin{equation*}\label{eqw:beta,sigma_z}
\beta(w, z)=\beta(\sigma_z(w),\sigma_z(z))=\beta(\sigma_z(w), \bfi)
\end{equation*}
for all $z,w\in \calU$, where $\sigma_z$ is as in Subsection 2.1. Hence $D(z,r) = \sigma_z^{-1}(D(\bfi,r))$. It follows that
\begin{align*}
|D(z,r)| ~=~  \int\limits_{\sigma_z^{-1}(D(\bfi,r))} dV(w)
~=~ \int\limits_{D(\bfi,r)} \left|\left(J_{\mathbb{C}} \sigma_z^{-1}\right)(u)\right|^2 dV(u).
\end{align*}
Combining with \eqref{eqn:jacobian}, this gives \eqref{eqn:vols}.

It remains to show that
\begin{equation*}
|D(\bfi,r)| ~=~ \frac {4\pi^n}{n!} \frac { \tanh^{2n} r } {(1-\tanh^2 r)^{n+1}}.
\end{equation*}
Note that $D(\bfi, r) = \Phi(B(0, R))$, where $R:=\tanh(r)$ and $\Phi$ is the Cayley transform. Thus,
\begin{eqnarray*}
|D(\bfi,r)| &=& \int\limits_{\Phi(B(0, R))} dV(w) \\
&\xlongequal{\xi=\Phi^{-1}(w)}& \int\limits_{B(0, R)} \frac {4}{|1+\xi_n|^{2(n+1)}} dV(\xi) \\
&\xlongequal{\zeta=\xi/R}& \int\limits_{\mathbb{B}} \frac {4 R^{2n}}{|1 -  (- Re_n) \cdot\overline{\zeta} |^{2(n+1)}} dV(\zeta) \\
&=& \frac {4\pi^n}{n!} \frac {R^{2n} } {(1-R^2)^{n+1}},
\end{eqnarray*}
where the last equality follows a simple calculation, see for instance \cite[p.263, (2.12)]{Liu15}.
The proof is complete.
\end{proof}

\begin{corollary}\label{cor:mu_r}
For any $z\in\mathcal{U}$ and $r>0$, the averaging function (defined as in \eqref{eqn:averaging})
\[
\widehat{\mu}_r (z) ~=~ \frac {n!}{4\pi^n} \frac {(1-\tanh^2 r)^{n+1}} {\tanh^{2n} r}
\frac {\mu(D(z,r))} {\bfrho(z)^{n+1}}.
\]
\end{corollary}



\begin{lemma}\label{blem}
Given $r > 0$, the inequalities
\begin{equation}\label{eqn:eqvltquan}
\frac {1-\tanh (r)}{1+ \tanh (r)} ~\leq~ \frac{|\bfrho(z,u)|}{|\bfrho(z,v)|}
~\leq~ \frac {1+\tanh (r)}{1-\tanh (r)}
\end{equation}
hold for all $z,u,v\in \calU$ with $\beta(u,v)\leq r$.
\end{lemma}


\begin{proof}

Let $\eta=\Phi^{-1}(z)$, $\xi=\Phi^{-1}(u)$ and $\zeta=\Phi^{-1}(v)$. We prove only the second inequality; then the first one follows by symmetry. By (\ref{eq:bergmanmetric}), then we have
\[
\beta_{\ball}(\xi,\zeta) = \beta(u,v) \leq \tanh(r).
\]
Also,
\begin{equation*}
\frac {\bfrho(z,u)} {\bfrho(z,v)} = \frac {(1- \eta\cdot\overline{\xi}) (1+\bar{\zeta}_n)}
{(1- \eta\cdot\overline{\zeta}) (1+\bar{\xi}_n)}.
\end{equation*}
Let $\tilde{\eta}= \varphi_{\xi}(\eta)$ and $\tilde{\zeta}= \varphi_{\xi}(\zeta)$. Then again by (\ref{eq:bergmanmetric}), $|\tilde{\zeta}|\leq \tanh(r) \in (0,1)$. Appealing to (\ref{eqn:moeb0}), we have
\begin{align*}
1- \eta\cdot\overline{\xi} ~=~& 1- \varphi_{\xi}(\tilde{\eta})\cdot\overline{ \varphi_{\xi}(0)}
~=~ \frac {1-|\xi|^2}{1- \tilde{\eta}\cdot\overline{\xi}}\\
\intertext{and}
1- \eta\cdot\overline{\zeta} ~=~& 1- \varphi_{\xi}(\tilde{\eta})\cdot\overline{ \varphi_{\xi}(\tilde{\zeta})}
~=~ \frac {(1-|\xi|^2) (1-\tilde{\eta}\cdot\overline{\tilde{\zeta}})}
{(1- \tilde{\eta}\cdot\overline{\xi}) (1- \xi\cdot\overline{\tilde{\zeta}})}.
\end{align*}
Thus we get
\[
\frac {1-\eta\cdot\overline{\xi}} {1- \eta\cdot\overline{\zeta}} =
\frac {1- \xi\cdot\overline{\tilde{\zeta}}} {1- \tilde{\eta}\cdot\overline{\tilde{\zeta}}}.
\]
Likewise,
\[
\frac {1+\bar{\zeta}_n} {1+\bar{\xi}_n} = \frac {1- \tilde{\varpi}\cdot\overline{\tilde{\zeta}}}
 {1- \xi\cdot\overline{\tilde{\zeta}}},
\]
where $\tilde{\varpi} := \varphi_{\xi} (-e_n)$. Hence,
\[
\frac {\bfrho(z,u)} {\bfrho(z,v)} -1  ~=~ \frac { \tilde{\eta}\cdot\overline{\tilde{\zeta}}-\tilde{\varpi}\cdot\overline{\tilde{\zeta}} } {(1- \tilde{\eta}\cdot\overline{\tilde{\zeta}})}.
\]
Since $|\tilde{\zeta}|\leq \tanh(r)$, we have
\[
 |1- \tilde{\eta}\cdot\overline{\tilde{\zeta}}| ~\geq~ 1-\tanh(r)
\]
and
\[
| \tilde{\eta}\cdot\overline{\tilde{\zeta}}-\tilde{\varpi}\cdot\overline{\tilde{\zeta}} |
~\leq~ 2 |\tilde{\zeta}| ~\leq~ 2\tanh(r).
\]
It follows that
\[
\left| \frac{\bfrho(z,u)} {\bfrho(z,v)} -1\right| 	~\leq~\frac {2\tanh (r)}{1-\tanh (r)},
\]
which implies the asserted inequality.
\end{proof}

In what follows we use the notation
\[
Q_j ~:=~ \overline{D(\bfi, j)}
\]
for $j=1,2\ldots$. Note that $Q_j \subset\subset Q_{j+1}$ for all $j\in \mathbb{N}$ and $\calU = \bigcup_{j=1}^{\infty} Q_j$.

\begin{lemma}\label{lem:sup Q_j}
Given $\lambda\in \mathbb{R}$ and $j\in \mathbb{N}$, there is a constant $C=C(n,\lambda,j)>0$ such that
\begin{equation}
\sup_{w\in Q_j} \frac {\bfrho(w)^{\lambda}} {|\bfrho(z,w)|^{n+1+\lambda}}
~\leq~ \frac {C}{|\bfrho(z,\bfi)|^{n+1+\lambda}}
\end{equation}
for all $z\in \calU$.
\end{lemma}

\begin{proof}
Let $w\in Q_j$ be fixed. We first show that
\begin{equation}\label{eqn:bscest2}
\frac {1-\tanh^2 (j)}{4} ~\leq~ \bfrho(w) ~\leq~ \frac {4}{1-\tanh^2 (j)}.
\end{equation}
Note from \eqref{eqn:tanh_beta} that
\[
\bfrho(w) ~=~ |\bfrho(w,\bfi)|^2 \left(1- \tanh^2 \beta(w,\bfi)\right).
\]
Since $\beta(w,\bfi)<j$ and $|\bfrho(w,\bfi)|^2 > 1/4$, the first inequality in \eqref{eqn:bscest2} follows.
Again, it follows from  \eqref{eqn:tanh_beta} that
\[
\bfrho(w) ~=~ \left(\frac {\bfrho(w)}{|\bfrho(w,\bfi)|}\right)^2 \frac {1}{1- \tanh^2 \beta(w,\bfi)}.
\]
The second inequality in \eqref{eqn:bscest2} is then immediate, in view of \eqref{eqn:elemtryeq1}.

Now,  the assertion of the lemma follows from \eqref{eqn:bscest2} and Lemma \ref{blem}.
\end{proof}

%
%

\subsection{Growth rate for functions in $A^p(\calU)$}

\begin{lemma}\label{lem:similar.Zhu 2.24}
Soppose $r>0$ and $p>0$. Then there exists a positive constant $C$ depending on $r$ such that
\[
|f(z)|^p \leq \frac{C}{\bfrho(z)^{n+1}}\int\limits_{D(z,r)} |f(w)|^p dV(w)
\]
for all $f\in H(\calU)$ and all $z\in\calU$.
\end{lemma}
\begin{proof}
Let $f\in H(\calU)$. Then $f\circ\Phi\in H(\ball)$.
Note that $D(\bfi,r)=\Phi(B(0,R))$ with $R=\tanh(r)$.
By the subharmonicity of $|f|^p$, we get
\begin{eqnarray*}
|f(\Phi(0))|^p &\leq& \frac{n!}{\pi^n R^{2n}}\int\limits_{B(0,R)} |f(\Phi(\xi))|^p dV(\xi)\\
&\xlongequal{w=\Phi(\xi)}& \frac{n!}{4\pi^n R^{2n}} \int\limits_{D(\bfi,r)} |f(w)|^p \frac{1}{|\bfrho(w,\bfi)|^{2(n+1)}}dV(w).
\end{eqnarray*}
Note that $f(\Phi(0))=f(\bfi)$ and $\inf\{|\bfrho(w,\bfi)|: w\in \calU\}\geq 1/2$. Then we have
\[
|f(\bfi)|^p ~\leq~ C \int\limits_{D(\bfi,r)} |f(w)|^p dV(w),
\]
with $C:=4^n n!/(\pi^n R^{2n})$. Replacing $f$ by $f\circ\sigma_z^{-1}$ in the above inequality,
we arrive at
\begin{eqnarray*}
|f(z)|^p &\leq& C \int\limits_{D(\bfi,r)} |f(\sigma_z^{-1}(w))|^p dV(w)\\
&\xlongequal{u=\sigma_z^{-1}(w)} & \frac{C}{\bfrho(z)^{n+1}}\int\limits_{D(z,r)} |f(u)|^p dV(u).
\end{eqnarray*}
This completes the proof of the Lemma.
\end{proof}

\begin{corollary}\label{cor:similar.Zhu 2.24}
Suppose $0<p<\infty$. Then
\begin{equation*}\label{eq:f,grow,speed}
|f(z)|~\leq~  \left(\frac{4^n n!}{\pi^n} \right)^{1/p} \frac{\|f\|_{p}}{ \bfrho(z)^{(n+1)/p}},\ \ \ z\in\calU
\end{equation*}
for all $f\in A^{p}(\calU)$.
\end{corollary}

\subsection{Weak convergence in $A^p(\calU)$}

\begin{lemma}\label{lem:weak,convergence}
Assume $\{f_j\}$ is a sequence in $A^p(\calU)$ with $1<p<\infty$. Then $f_j\to 0$ weakly in $A^p(\calU)$ if and only if  $\{f_j\}$ is bounded in $A^p(\calU)$ and converges to $0$ uniformly on each compact subset of $\calU$.
\end{lemma}

\begin{proof}
The proof of the sufficiency is not hard. The proof of the necessity is also a standard normal family argument.
We include a proof for reader's convenience. Suppose $\{f_j\}$ converges to $0$ weakly in $A^p(\calU)$. Then,
by the uniform boundedness principle, $\{f_j\}$ is bounded in $A^p(\calU)$. Together with Corollary \ref{cor:similar.Zhu 2.24},
this implies that $\{f_j\}$  is uniformly bounded on each compact subset of $\calU$ and thus is a normal family. Note that $f_j\to 0$ pointwise by the reproducing property of the kernel $K_z\in A^{p^{\prime}}(\calU)$ (see \cite[Theorem 2.1]{DK93}) with $p^{\prime}=p/(p-1)$. (In fact, by \eqref{eqn:keylem}, for every $z\in\calU$, $K_z$ belongs to $A^t(\calU)$ for all $1<t<\infty$.)
Now, by a standard argument, we see that $f_j\to 0$ uniformly on each compact subset of $\calU$. The proof is complete.

\end{proof}

\begin{lemma}\label{lem:K_z,weak}
For $1<p<\infty$, we have $K_z \|K_z\|_{p}^{-1}\to 0$ weakly in $A^p(\calU)$ as $z\to\partial \widehat{\calU}$.
\end{lemma}
\begin{proof}

In view of lemma \ref{lem:weak,convergence}, it suffices to prove that $K_z \|K_z\|_{p}^{-1}$ converges to $0$ uniformly on every $Q_j$.

By Lemmas \ref{lem:normofbergman} and \ref{blem}, there exists a constant $C>0$ such that
\[
 \sup_{w\in Q_j} \left| \frac {K_z(w)} {\|K_z\|_{p}} \right| ~\leq~ C\; \frac{\bfrho(z)^{(n+1)/p^{\prime}}}
 {|\bfrho(z,\bfi)|^{n+1}}
\]
for all $z\in \calU$. Since $2|\bfrho(z,\bfi)| = |z_n+i| \geq 1$ for all $z\in\calU$, we have
\[
 \sup_{w\in Q_j} \left| \frac {K_z(w)} {\|K_z\|_{p}} \right| ~\leq~  C \bfrho(z)^{(n+1)/p^{\prime}},
\]
which implies that $K_z \|K_z\|_{p}^{-1}\to 0$ uniformly on $Q_j$ as $z\to b\calU$.
On the other hand, by \eqref{eqn:elemtryeq1} and the fact that $2|\bfrho(z,\bfi)| \geq |z|$ for all $z\in\calU$,
\[
 \sup_{w\in Q_j} \left| \frac {K_z(w)} {\|K_z\|_{p}} \right| ~\leq~   \frac {C}{|z|^{(n+1)/p}},
\]
 which implies that $K_z \|K_z\|_{p}^{-1}\to 0$ uniformly on $Q_j$ as $|z|\to\infty$. The proof of the lemma is complete.
\end{proof}

\section{Carleson measures}



\begin{theorem}\label{thm:Carleson_measure}
Suppose $0<p<\infty$, $r>0$ and $\mu$ is a positive Borel measure.
Then the following conditions are equivalent:
\begin{enumerate}
\item[(a)]
$\mu$ is a Carleson measure for $A^p(\calU)$.

\item[(b)]
There exists a constant $C>0$ such that
\[
\int\limits_{\calU} \frac{\bfrho(a)^{n+1}}{|\bfrho(z,a)|^{2(n+1)}} d\mu(z)\leq C
\]
for all $a\in\calU$.

\item[(c)]
There exists a constant $C>0$ such that
\[
\mu(D(a,r))\leq C \bfrho(a)^{n+1}
\]
for all $a\in\calU$.

\item[(d)]
There exists a constant $C>0$ such that
\[
\mu(D(a_k,r))\leq C \bfrho(a_k)^{n+1}
\]
for all $k\geq 1$, where $\{a_k\}$ is an $r$-lattice in the Bergman metric.
\end{enumerate}
\end{theorem}

\begin{proof}
It is easy to see that (a) implies (b). In fact, setting
\[
f(z)=\left[ \frac{\bfrho(a)^{n+1}}{\bfrho(z,a)^{2(n+1)}} \right]^{1/p}
\]
in (a) immediately yields (b).

If (b) is true, then
\[
\int\limits_{D(a,r)} \frac{\bfrho(a)^{n+1}}{|\bfrho(z,a)|^{2(n+1)}} d\mu(z)\leq C
\]
for all $a\in\calU$. This along with Lemma \ref{blem} shows that (c) must be true.

That (c) implies (d) is trivial.

It remains to prove that (d) implies (a). So we assume that there exists a constant $C_1>0$ such that
\[
\mu(D(a_k,r))\leq C_1 \bfrho(a_k)^{n+1}
\]
for all $k\geq 1$. If $f$ is holomorphic in $\calU$, then
\begin{eqnarray*}
\int\limits_{\calU} |f|^p d\mu &\leq& \sum_{k=1}^{\infty}\int\limits_{D(a_k,r)} |f(z)|^p d\mu(z)\\
&\leq& \sum_{k=1}^{\infty} \mu(D(a_k,r)) \sup\{ |f(z)|^p: z\in D(a_k,r)\}.
\end{eqnarray*}
By Lemmas \ref{lem:similar.Zhu 2.24} and \ref{blem}, there exists a constant $C_2>0$ such that
\[
\sup\{ |f(z)|^p: z\in D(a_k,r)\} \leq  \frac{C_2}{\bfrho(a_k)^{n+1}}\int\limits_{D(a_k,2r)} |f(w)|^p dV(w)
\]
for all $k\geq 1$. It follows that
\begin{eqnarray*}
\int\limits_{\calU} |f|^p d\mu &\leq& C_2 \sum_{k=1}^{\infty} \frac{\mu(D(a_k,r))}{\bfrho(a_k)^{n+1}} \int\limits_{D(a_k,2r)} |f(w)|^p dV(w)\\
&\leq& C_1 C_2 \sum_{k=1}^{\infty}  \int\limits_{D(a_k,2r)} |f(w)|^p dV(w)\\
&\leq& C_1 C_2 N  \int\limits_{\calU} |f(w)|^p dV(w),
\end{eqnarray*}
where $N$ is as in Lemma \ref{lem:decomposition}. This completes the proof of the theorem.
\end{proof}

It follows from the above theorem that the property of being a Carleson measure for $A^p(\calU)$
is independent of $p$, so that if $\mu$ is a Carleson measure for $A^p(\calU)$ for some $p$,
then $\mu$ is a Carleson measure for $A^p(\calU)$ for all $p$.



\begin{theorem}\label{thm:vanishing_Carleson_measure}
Suppose $1<p<\infty$, $r>0$ and $\mu$ is a positive Borel measure. Then the following conditions are equivalent:
\begin{enumerate}
\item[(a)]
$\mu$ is a vanishing Carleson measure for $A^p(\calU)$.
\item[(b)]
The measure $\mu$ satisfies
\[
\lim _{a\to\partial \widehat{\calU}} \, \int\limits_{\calU} \frac{\bfrho(a)^{n+1}}{|\bfrho(z,a)|^{2(n+1)}} d\mu(z)=0
\]
\item[(c)]
The measure $\mu$ has the property that
\[
\lim_{a\to \partial \widehat{\calU}} \frac{\mu(D(a,r))}{\bfrho(a)^{n+1}}=0.
\]
\item[(d)]
For $\{a_k\}$ an $r$-lattice in the Bergman metric, we have
\[
\lim_{k\to\infty} \frac{\mu(D(a_k,r))}{\bfrho(a_k)^{n+1}}=0.
\]
\end{enumerate}
\end{theorem}

\begin{proof}
If (a) is true, it means that the inclusion map $i_p$ is compact.
By Lemma \ref{lem:weak,convergence} and Lemma \ref{lem:K_z,weak}, we see that $k_a$ converges to $0$ uniformly on each compact subset of $\calU$ as $a\to \partial \widehat{\calU}$, and so does $g_a:=k_a^{2/p}$. It is obvious that $\|g_a\|_p=1$. Then by Lemma \ref{lem:weak,convergence} again, $g_a$ converges to 0 weakly in $A^p(\calU)$ as $a\to\partial \widehat{\calU}$. Therefore,
\[
 \int\limits_{\calU}\frac{\bfrho(a)^{n+1}}{|\bfrho(z,a)|^{2(n+1)}} d\mu(z) = C \int\limits_{\calU} |g_a(z)|^p d\mu(z) \to 0
\]
as $a\to\partial \widehat{\calU}$. So (b) follows.

If (b) holds, then
\[
\lim _{a\to\partial \widehat{\calU}} \, \int\limits_{D(a,r)} \frac{\bfrho(a)^{n+1}}{|\bfrho(z,a)|^{2(n+1)}} d\mu(z)=0.
\]
By Lemma \ref{blem}, $|\bfrho(z,a)|$ and $\bfrho(z)$ are comparable when $z\in D(a,r)$. So (c) follows.

Note that $a_k\to \partial \widehat{\calU}$ as $k\to\infty$ if $\{a_k\}$ is an $r$-lattice in the Bergman metric. That (c) implies (d) is immediate.

It remains to prove that (d) implies (a). Assume that (d) is true and $\{f_j\}$ is a sequence in $A^p(\calU)$ that converges to $0$ weakly. We only need to prove that
\begin{equation}\label{eq:weak f_j}
\lim_{j\to \infty} \int\limits_{\calU} |f_j(z)|^p d\mu(z)=0.
\end{equation}
By Lemma \ref{lem:weak,convergence}, $\{f_j\}$ has the property that $\sup_j \|f_j\|_{p}\leq M$ for some positive constant $M$ and converges to $0$ uniformly on each compact subset of $\calU$. By assumption, given $\varepsilon>0$ there exists a positive integer $N_0$ such that
\[
\frac{\mu(D(a_k,r))}{\bfrho(a_k)^{n+1}} <\varepsilon \quad\text{for all }\;  k\geq N_0.
\]
By the last part of the proof of Theorem \ref{thm:Carleson_measure}, there is a constant $C>0$ such that
\begin{align*}
\sum_{k=N_0}^{\infty} &\int\limits_{D(a_k,r)} |f_j(w)|^p d\mu(w)\\
&\leq C \sum_{k=N_0}^{\infty} \frac{\mu(D(a_k,r))}{\bfrho(a_k)^{n+1}} \int\limits_{D(a_k,2r)} |f_j(w)|^p dV(w)\\
&\leq \varepsilon C N \int\limits_{\calU} |f_j(w)|^p dV(w)
\leq \varepsilon C N M^p
\end{align*}
for all $j$, where $C$, $N$, and $M$ are all independent of $\varepsilon$. Since
\[
\lim_{j\to \infty} \sum_{k=1}^{N_0-1} \int\limits_{D(a_k,r)} |f_j(z)|^p d\mu(z)=0
\]
by uniform convergence. Therefore,
\begin{eqnarray*}
&&\limsup_{j\to\infty} \int\limits_{\calU} |f_j(z)|^p d\mu(z)\\
&\leq& \limsup_{j\to\infty} \left[ \sum_{k=1}^{N_0-1} \int\limits_{D(a_k,r)} |f_j(z)|^p d\mu(z) +
\sum_{k=N_0}^{\infty} \int\limits_{D(a_k,r)} |f_j(w)|^p dV(w)\right]\\
&\leq& \varepsilon C N M^p.
\end{eqnarray*}
Since $\varepsilon$ is arbitrary, (\ref{eq:weak f_j}) follows. The proof of the theorem is complete.
\end{proof}

It follows from the above theorem that the property of being a vanishing Carleson measure for $A^p(\calU)$ depends neither on $p$ nor on $r$.

\section{Dense subspaces of $A^p(\calU)$}\label{section:dense subspace}

Given $\alpha$ real, we denote by $\calS_{\alpha}$ the vector space of functions $f$ holomorphic in $\calU$ satisfying
\[
\sup_{z\in\calU} |z_n+i|^{\alpha} |f(z)| ~<~ \infty.
\]

%

\begin{theorem}\label{thm:dense,subspace}
If $1\leq  p < \infty$ and $\alpha > n+1/p$, then $\calS_{\alpha}$ is a dense subspace of $A^p(\calU)$.
\end{theorem}

\begin{proof}

It is immediate from \eqref{eqn:keylem} that $\calS_{\alpha}$ is contained in $A^p(\calU)$ whenever $\alpha > n+1/p$.

We now prove the density of $\calS_{\alpha}$ in $A^p(\calU)$. Let $f$ be arbitrary in $A^p(\calU)$.
Put $f_j=f\cdot \chi_{Q_j}$ for $j=1,2,\ldots$, where $Q_j ~:=~ \overline{D(\bfi, j)}$ and $\chi_{Q_j}$ is the characteristic function of $Q_j$.
Clearly, $\|f_j-f\|_{p} \to 0$ as $j\to \infty$.

Given $\lambda>-1$, let $P_{\lambda}$ be the integral operator given by
\[
P_{\lambda}g(z)= c_{\lambda} \int\limits_{\calU} \frac{\bfrho(w)^{\lambda}}{\bfrho(z,w)^{n+1+\lambda}} g(w) dV(w),
\quad z\in \calU,
\]
where $c_{\lambda}=\Gamma (n+1+\lambda)/(4\pi^n\Gamma (1+\lambda))$.
It was shown in \cite[Theorem 3.1]{DK93} that $P_{\lambda}$ is a bounded projection from $L^p(\calU)$ onto
$A^p(\calU)$, provided that $\lambda>1/p-1$.


Take $\lambda=\alpha-n-1$. By H\"older's inequality, we obtain
\[
|P_{\alpha-n-1} f_j (z)|
~\leq~ c_{\alpha-n-1} \|f\|_{p} \, |Q_j|^{1/p^{\prime}} \sup_{w\in Q_j} \frac {\bfrho(w)^{\alpha-n-1}}
{|\bfrho(z,w)|^{\alpha}}
\]
for all $z\in \calU$, where $p^{\prime}=p/(p-1)$ and $|Q_j|$ stands for the Lebesgue measure of $Q_j$. Combining this inequality with Lemma \ref{lem:sup Q_j}, we obtain
\[
|P_{\alpha-n-1} f_j (z)| ~\leq~ \frac {C \|f\|_{p}} {|\bfrho(z,\bfi)|^{\alpha}}
\]
for all $z\in \calU$, where $C>0$ is a constant depending on $n$, $\alpha$, $j$ and $p$. Thus, $P_{\alpha-n-1} f_j\in \calS_{\alpha}$.

Since $f\in A^p(\calU)$ and $P_{\alpha-n-1}$ is a bounded projection from $L^p(\calU)$ onto
$A^p(\calU)$,
\[
\|P_{\alpha-n-1} f_j - f\|_p = \|P_{\alpha-n-1} (f_j - f)\|_p \leq \|P_{\alpha-n-1}\|\, \|f_j-f\|_{p} \to 0
\]
as $j\to \infty$. This implies that $\calS_{\alpha}$ is dense in $A^p(\calU)$.
\end{proof}


\begin{corollary}\label{cor:densely}
The Toeplitz operator $T_{\mu}$ with symbol $\mu\in \mathcal{M}_{+}$ is densely defined on $A^p(\calU)$ for every $1<p<\infty$.
\end{corollary}

\begin{proof}
It suffices to show that
\[
\int\limits_{\calU} |K(z,w) f(w)| d\mu(w) < \infty
\]
holds for every $f\in \calS_{\alpha}$ with $\alpha>0$, and for each fixed $z\in \calU$. Indeed, it follows by \eqref{eqn:elemtryeq1} that there exists a constant $C>0$ depending on $n$ and $\alpha$ such that
\[
\int\limits_{\calU} |K(z,w) f(w)| d\mu(w) \leq C  \bfrho(z)^{-n-1} \int\limits_{\calU} \frac {d\mu(w)}
{|w_n+i|^{\alpha}}<\infty,
\]
as desired.
\end{proof}

\section{Proofs of the theorems}

Just like the cases of the unit disk or bounded symmetric domains, the key step is the justification of the equality
\begin{equation}\label{eq:T_mu}
\langle T_{\mu}f,g \rangle = \int\limits_{\calU} f \overline{g} d\mu,
\end{equation}
where $\langle\cdot,\cdot\rangle$ denotes the duality pairing between $A^p(\calU)$ and $A^{p^{\prime}}(\calU)$.
This would enable us to make a connection between Carleson measures and positive Toeplitz operators.


\begin{lemma}
Let $0<\alpha<n+1$. Then there exists a constant $C>0$ such that
\begin{equation}\label{eqn:rho^n+1 alpha}
\int\limits_{\calU} \frac{dV(u)} {|\bfrho(z,u)|^{n+1} |\bfrho(u,\bfi)|^{\alpha}}
~\leq~ \frac {C}  {|\bfrho(z,\bfi)|^{\alpha}} \left(1+\log \frac {|\bfrho(z,\bfi)|^2}{\bfrho(z)}\right)
\end{equation}
for any $z\in\calU$.
\end{lemma}

\begin{proof}
Given $z\in\calU$, let $\eta :=\Phi^{-1}(z)$. Making the change of variables $u=\Phi(\xi)$ in the integral
and using Lemma \ref{lem:cayley}, we obtain
\[
\int\limits_{\calU} \frac{dV(u)}{|\bfrho(z,u)|^{n+1} |\bfrho(u,\bfi)|^{\alpha}}
~=~ 4|1+\eta_n|^{n+1} \int\limits_{\ball} \frac{dV(\xi)}{ |1-\eta\cdot\overline{\xi}|^{n+1} |1+\xi_n|^{n+1-\alpha}}.
\]
By \cite[Theorem 3.1]{Zhang18}, the last integral is dominated by a constant multiple of
\[
\frac {1}{|1+\eta_n|^{n+1-\alpha}} \log\frac{e}{|1-\eta\cdot\overline{\varphi_{\eta}(-e_n)}|}
~=~ \frac {1}{|1+\eta_n|^{n+1-\alpha}} \left(1+ \log\frac{|1+\eta_n|}{1-|\eta|^2}\right),
\]
where $e_n:=(0^{\prime},1)$. Thus, there exists a constant $C>0$ such that
\begin{align*}
\int\limits_{\calU} \frac{dV(u)}{|\bfrho(z,u)|^{n+1} |\bfrho(u,\bfi)|^{\alpha}}
~\leq~& C |1+\eta_n|^{\alpha} \left(1+ \log \frac {1}{1-|\eta|^2}\right)\\
~=~& \frac {C}  {|\bfrho(z,\bfi)|^{\alpha}} \left(1+\log \frac {|\bfrho(z,\bfi)|^2}{\bfrho(z)}\right)
\end{align*}
as desired.
\end{proof}

\begin{lemma}\label{lem:T_mu:S_alpha to A^p}
Suppose that $1<p<\infty$, $n+1/p<\alpha<n+1$ and that $\mu\in\mathcal{M}_+$ is a Carleson measure for $A^1(\calU)$. Then $T_{\mu}$ maps $\calS_{\alpha}$ into $A^p(\calU)$.
\end{lemma}

\begin{proof}
The proof of Corollary \ref{cor:densely} shows that $T_{\mu}$ is well defined on $\calS_{\alpha}$.
Let $f\in\calS_{\alpha}$. We proceed to show that $T_{\mu} f$ in $A^p(\calU)$.
Since $\mu$ is a Carleson measure for $A^1(\calU)$, there exists a constant $C>0$ such that
\begin{align*}
\int\limits_{\calU} |K(z,w) f(w)| d\mu(w) ~\leq~& C \int\limits_{\calU} |K(z,w) f(w)| dV(w)\\
\leq~& C \int\limits_{\calU} \frac{dV(u)} {|\bfrho(z,u)|^{n+1} |\bfrho(u,\bfi)|^{\alpha}}.
\end{align*}
This, together with \eqref{eqn:rho^n+1 alpha}, implies
\begin{equation}\label{eqn:peToep}
|T_{\mu} f(z)| ~\leq~ \frac {C}  {|\bfrho(z,\bfi)|^{\alpha}} \left(1+\log \frac {|\bfrho(z,\bfi)|^2}{\bfrho(z)}\right)
\end{equation}
for all $z\in\calU$.
Since $\log x<x^{\epsilon}$ holds for any $x>1$ and any $\epsilon>0$, we have
\begin{equation}\label{eqn:logx<x^epsilon}
 \log\frac{|\bfrho(z,\bfi)|^2}{\bfrho(z)}~ <~  \frac {|\bfrho(z,\bfi)|^{2\epsilon}} {\bfrho(z)^{\epsilon}}
\end{equation}
for all $z\in\calU$. Choose $\epsilon$ small enough so that $0<\epsilon<\min\{1/p,\alpha-(n+1)/p\}$. Then we have
\[
\int\limits_{\calU} |T_{\mu}f(z)|^p dV(z) ~\leq~ C\left( \int\limits_{\calU}\frac{dV(z)}{|\bfrho(z,\bfi)|^{p\alpha}} +\int\limits_{\calU}
\frac{\bfrho(z)^{-p\epsilon}}{ |\bfrho(z,\bfi)|^{p(\alpha-2\epsilon)}} dV(z)\right),
\]
which along with \eqref{eqn:keylem} completes the proof.
\end{proof}

\begin{lemma}\label{lem:T_mu}
Suppose that $1<p<\infty$, $n+1/p<\alpha<n+1$, $\gamma>n+(p-1)/p$ and that
$\mu\in \mathcal{M}_+$  is a Carleson measure for $A^1(\calU)$.
Then \eqref{eq:T_mu} holds for all $f\in \calS_{\alpha}$ and $g\in\calS_{\gamma}$.
\end{lemma}

\begin{proof}
Let $f\in \calS_{\alpha}$ and $g\in \calS_{\gamma}$. In view of Theorem \ref{thm:dense,subspace} and Lemma \ref{lem:T_mu:S_alpha to A^p},
we see that $f\in A^p(\calU)$, $g\in A^{p^{\prime}}(\calU)$ and $T_{\mu} f\in A^p(\calU)$.
Then both side of \eqref{eq:T_mu} are well defined.
By Fubini's theorem,
\begin{eqnarray*}
\langle T_{\mu}f,g \rangle &=& \int\limits_{\calU} \bigg(\int\limits_{\calU} K(z,w) f(w) d\mu(w)\bigg) \overline{g(z)} dV(z)\\
&=& \int\limits_{\calU} f(w) \overline{\bigg(\int\limits_{\calU} K (w,z) g(z) dV(z)\bigg)}  d\mu(w)\\
&=& \int\limits_{\calU} f(w) \overline{g(w)} d\mu(w),
\end{eqnarray*}
where the last equality follows from \cite[Theorem 2.1]{DK93}.  The interchange of the
order of integration is justified as follows.
By \eqref{eqn:peToep} and \eqref{eqn:logx<x^epsilon},with $\epsilon\in (0,1)$, we obtain
\begin{align*}
\int\limits_{\calU} & \bigg(\int\limits_{\calU} |K(z,w) f(w) g(z)| d\mu(w)\bigg) dV(z)\\
&\leq~ C \int\limits_{\calU}  \frac {1}  {|\bfrho(z,\bfi)|^{\alpha+\gamma}} \left(1+  \frac {|\bfrho(z,\bfi)|^{2\epsilon}}
{\bfrho(z)^{\epsilon}}\right)  dV(z)\\
&\leq~ C\, \bigg( \int\limits_{\calU} \frac{dV(z)}{|\bfrho(z,\bfi)|^{\alpha+\gamma}}
+\int\limits_{\calU} \frac{\bfrho(z)^{-\epsilon}}{ |\bfrho(z,\bfi)|^{\alpha+\gamma-2\epsilon}} dV(z)\bigg),
\end{align*}
which is finite, in view of \eqref{eqn:keylem}. The proof of the lemma is complete.
\end{proof}


\begin{corollary}\label{cor:T_mu,bounded}
Suppose that $\mu\in \mathcal{M}_+$ is a Carleson measure for $A^{q}(\calU)$ for some $q>0$. Then $T_{\mu}$ is
densely defined and extends to a bounded operator on $A^{p}(\calU)$ for any $p>1$. Moreover, \eqref{eq:T_mu}
holds for all $f\in A^p(\calU)$ and $g\in A^{p^{\prime}}(\calU)$, where $p^{\prime}=p/(p-1)$.
\end{corollary}

Now we are in the position to prove Theorems \ref{thm:T_mu,bounded} and \ref{thm:T_mu,compact}.

\begin{proof}[Proof of Theorem \ref{thm:T_mu,bounded}.]

Combining Theorem \ref{thm:Carleson_measure} with Corollary \ref{cor:mu_r}, we see that (\romannumeral2), (\romannumeral3) and (\romannumeral4)
are equivalent. Also, that (\romannumeral4) implies (\romannumeral1) is immediate from Corollary \ref{cor:T_mu,bounded}.
So we only need to prove that (\romannumeral1) implies (\romannumeral2).

Assume that $T_{\mu}$ is bounded on $A^p(\calU)$. For every $z\in\calU$, by Lemma \ref{lem:normofbergman}, $\langle T_{\mu}k_z,k_z \rangle$ is well defined.
Lemma \ref{lem:normofbergman} also yields the following identity:
\[
\|K_z\|_{p} \|K_z\|_{p^{\prime}} = C K(z,z),
\]
where $C$ is a positive constant depending on $p$ and $n$. Hence
\begin{align*}\label{eq:widetilde,mu}
 |\langle T_{\mu}k_z,k_z \rangle | ~\leq~ \frac {\|T_{\mu}K_z\|_{p} \|K_z\|_{p^{\prime}}}{K(z,z)}
~=~ C  \left\|T_{\mu} \left(\frac {K_z}{\|K_z\|_{p}}\right) \right\|_{p}  ~\leq~ C \|T_{\mu}\|.
\end{align*}
On the other hand, again by \cite[Theorem 2.1]{DK93}, we have
\[
\langle T_{\mu}k_z,k_z \rangle ~=~ \frac{\langle T_{\mu}k_z,K_z \rangle}{\sqrt{K(z,z)}} ~=~ \frac{(T_{\mu} k_z)(z)}{\sqrt{K(z,z)}}
~=~\widetilde{\mu}(z).
\]
Hence, $\widetilde{\mu}$ is a bounded function on $\calU$.
%
%
\end{proof}


\begin{proof}[Proof of Theorem \ref{thm:T_mu,compact}.]
Combining Theorem \ref{thm:vanishing_Carleson_measure} with Corollary \ref{cor:mu_r},
we see that (\romannumeral2), (\romannumeral3) and (\romannumeral4) are equivalent.
Therefore, it will suffice to prove the implications (\romannumeral1) $\Rightarrow$ (\romannumeral2)
and  (\romannumeral4) $\Rightarrow$ (\romannumeral1).

(\romannumeral1) $\Rightarrow$ (\romannumeral2). Assume that $T_{\mu}$ is compact on $A^p(\calU)$ for some $p>1$.
As is shown in the proof of Theorem \ref{thm:T_mu,bounded},
\[
|\widetilde{\mu}(z)| ~\leq~ C  \left\|T_{\mu} \left(\frac {K_z}{\|K_z\|_{p}}\right) \right\|_{p}
\]
for all $z\in\calU$. This, together with Lemma \ref{lem:K_z,weak} and the compactness of $T_{\mu}$, implies $\widetilde{\mu} \in C_0(\calU)$.


(\romannumeral4) $\Rightarrow$ (\romannumeral1). Assume that $\mu$ is a vanishing Carleson measure for $A^q(\calU)$.
Then $\mu$ also is a vanishing Carleson measure for $A^{p^{\prime}}(\calU)$, where $p^{\prime}:=p/(p-1)$, hence
$\sup \big\{ \|g\|_{L^{p^{\prime}}(\mu)}: \|g\|_{p^{\prime}}=1 \big\}$ is finite. Also, by Theorem \ref{thm:T_mu,bounded},
$T_{\mu}$ is bounded on $A^p(\calU)$. Therefore, by Corollary \ref{cor:T_mu,bounded}, we have
\begin{align*}
\|T_{\mu} f\|_{p} ~=~&\sup \big\{ |\langle T_{\mu}f,g \rangle|: \|g\|_{p^{\prime}}=1 \big\} \\
~=~& \sup \left\{ \left|\int\limits_{\calU} f \overline{g} d\mu \right| : \|g\|_{p^{\prime}}=1 \right\}  \\
\leq~& \|f\|_{L^p(\mu)} \sup \big\{ \|g\|_{L^{p^{\prime}}(\mu)}: \|g\|_{p^{\prime}}=1 \big\}
\end{align*}
for any $f\in A^p(\calU)$. %
If $f_j\to 0$ weakly in $A^{p}(\calU)$, then the compactness of $i_p$ implies that $\|f_j\|_{L^p(\mu)}\to 0$,
and hence $\|T_{\mu} f_j \|_{p}\to 0$. This implies that $T_{\mu}$ is compact on $A^p(\calU)$. The proof of the theorem is complete.
\end{proof}


\end{document}